\documentclass[11pt,oneside,reqno]{amsart}
\usepackage{amsmath,amsfonts,amssymb}
\usepackage{esint}
\usepackage{amsthm} 
\usepackage{tikz}
\usepackage{dsfont} 
\usepackage{esint}

\usepackage[a4paper,left=30mm,right=30mm,top=35mm,bottom=35mm,marginpar=25mm]{geometry} 
\usetikzlibrary{shadings,intersections,patterns}

\usepackage{pgfplots}

\pgfplotsset{compat=1.10}
\usepgfplotslibrary{fillbetween}

\usepackage{xfrac}
\usepackage{xcolor}
\usetikzlibrary{positioning}
\usepackage{mathtools}

\usepackage[shortlabels]{enumitem}

\usepackage{hyperref} 
\usepackage[capitalise, nameinlink, noabbrev]{cleveref} 
\hypersetup{colorlinks=true, 
linkcolor=red,
citecolor=blue,
}

\theoremstyle{plain} 

\newtheorem{theorem}{Theorem}[section]

\newtheorem{corollary}[theorem]{Corollary}
\newtheorem{lemma}[theorem]{Lemma}
\newtheorem{definition}[theorem]{Definition}
\theoremstyle{remark}
\newtheorem{remark}[theorem]{Remark}

\newcommand{\R}{\mathbb{R}}

\newcommand{\N}{\mathbb{N}}

\newcommand{\eps}{\varepsilon}

\newcommand{\ind}{\mathds{1}}
\newcommand{\HH}{\mathcal{H}}

\DeclareMathOperator{\Sing}{Sing}

\makeindex

\def\XXint#1#2#3{{\setbox0=\hbox{$#1{#2#3}{\int}$}
     \vcenter{\hbox{$#2#3$}}\kern-.5\wd0}}

\numberwithin{equation}{section}

\title[Regularity of the vectorial free boundaries]{Rectifiability and almost everywhere uniqueness of the blow-up for the vectorial Bernoulli free boundaries}

\author[G.~De Philippis]{Guido De Philippis}
\address{\textit{G.~De Philippis:} Courant Institute of Mathematical Sciences, New York University, 251 Mercer St., New York, NY 10012, USA.}
\email{guido@cims.nyu.edu}

\author[M.~Engelstein]{Max Engelstein}
\address{\textit{M.~Engelstein:} Department of Mathematics, University of Minnesota, Minneapolis, MN, 55455}
\email{mengelst@umn.edu}

\author[L.~Spolaor]{Luca Spolaor}
\address{\textit{L.~Spolaor:} Department of Mathematics, University of California, San Diego, La Jolla, CA, 92093
 } \email{lspolaor@ucsd.edu}

\author[B.~Velichkov]{Bozhidar Velichkov}
\address{{\it B. Velichkov:} 
	Dipartimento di Matematica, Universit\`a di Pisa,
	Largo Bruno Pontecorvo, 5, 56127 Pisa - ITALY}
\email{bozhidar.velichkov@unipi.it}

\thanks{{\bf Acknowledgments.} 
M.E. has been partially supported by the NSF grant DMS 2000288. L.S. has been partially supported by the NSF grant DMS 1810645. B.V. was supported by the European Research Council (ERC), under the European Union’s Horizon 2020 research and innovation programme, through the project ERC VAREG - \it Variational approach to the regularity of the free boundaries \rm (grant agreement No. 853404).}

\subjclass[2010]{35R35, 49Q10, 47A75}

\keywords{free boundary, vectorial problem, singular set, rectifiability, stratification, uniqueness of the blow-up limits}

\begin{document}

\begin{abstract} We prove that for minimizers of the vectorial Alt-Caffarelli functional the two-phase singular set of the free boundary is rectifiable and the blow-up is unique almost everywhere on it. While the first conclusion is an application of the recent techniques developed by Naber and Valtorta, the uniqueness part follows from the rectifiability and a new application of the Alt-Caffarelli-Friedman monotonicity formula. 
\end{abstract}
\maketitle


\section{Introduction}

Let $\Omega\subset \R^d$ be an open subset of $\R^d$. We say that a function 
$$U=(u_1,\dots,u_k):\Omega\to\R^k$$ 
is in the Sobolev space $H^1(\Omega;\R^k)$, if $u_j\in H^1(\Omega)$ for every $j=1,\dots,k$, and we denote the Dirichlet integral of $U$ by 
$$\int_\Omega |\nabla U|^2\,dx:=\sum_{j=1}^k\int_\Omega |\nabla u_j|^2\,dx.$$
Moreover, we say that $U\in H^1_0(\Omega;\R^k)$ if $u_j\in H^1_0(\Omega)$ for every $j=1,\dots,k$. When $U\in H^1_0(\Omega;\R^k)$ we will automatically assume that $U$ is extended by zero outside $\Omega$.
We will denote by $|U|$ the norm of the vector $U$, that is, 
$$|U|=\Big(\sum_{j=1}^ku_j^2\Big)^{\sfrac12}.$$ 
Given a constant $\Lambda>0$, an open set $\Omega\subset\R^d$ and a function $U\in H^1(\Omega;\R^k)$, we define the vectorial Alt-Caffarelli functional as 
$$J(U;\Omega):=\int_\Omega \Big(|\nabla U|^2+\Lambda(x)\ind_{\{|U|>0\}}\Big)\,dx\,$$
where $\Lambda:\R^d\to\R$ is a positive $C^{0,\alpha}$-regular function bounded away from zero and infinity, that is, there is a constant $C_\Lambda>0$ such that
$$\frac1{C_\Lambda}\le \Lambda(x)\le C_\Lambda\quad\text{for every}\quad x\in \R^d.$$
For the sake of simplicity, throughout the paper we will assume that $\Lambda$ is a constant; the proofs in the general case require only minor standard technical modifications.\medskip

Given an open set $D\subset \R^d$, we say that the function $U:D\to\R^k$ is a local minimizer of $J$ in $D$ if $U\in H^1(\Omega;\R^k)$ for every open set $\Omega\Subset D$ and if 
$$J(U;\Omega)\le J(V;\Omega)\quad\text{for every}\quad V\in H^1(D;\R^k)\quad\text{such that}\quad U-V\in H^1_0(\Omega;\R^k).$$
The vectorial Bernoulli problem consists in minimizing the functional $J$ in some bounded open set $D$ among all functions $U\in H^1(D;\R^k)$ with prescribed boundary condition. In particular, given a minimizer $U$, we are interested in describing the local structure of the free boundary 
$\partial\Omega_U$ inside $D$, where $\Omega_U$ is the set
$$\Omega_U:=\{|U|>0\}.$$
This problem generalizes both the classical one-phase Bernoulli problem first studied by Alt and Caffarelli in \cite{altcaf} and the two-phase Bernoulli problem of Alt-Caffarelli-Friedman \cite{alcafr} for which the full regularity of the free boundary was obtained only recently in \cite{DSV}. \medskip

The vectorial Bernoulli problem was introduced simultaneously in \cite{CSY}, \cite{MTV1} and \cite{KL1}, and at first it was studied in the so-called {\it non-degenerate} case in which it is a priori known that at least one of the components of $U$ does not change sign. The regularity of the free boundary in the general {\it degenerate case} was first obtained in dimension $d=2$ (and for any $k\ge 2$) in \cite{spve}, where it was shown that $\partial\Omega_U$ can be decomposed is the disjoint union of two sets: the regular part $\text{\rm Reg}(\partial\Omega_U)$ is an open subset of $\partial\Omega_U$ and a smooth manifold, while the remaining singular set $\text{\rm Sing}(\partial\Omega_U)$ is contained in a countable union of $C^{1,\alpha}$-regular curves, the blow-up limit of $U$ being unique at every singular point. In higher dimensions, the $C^{1,\alpha}$ regularity of $\text{\rm Reg}(\partial\Omega_U)$  in the degenerate case was first obtained in \cite{KL2} and later in \cite{MTV2} (the precise definition of the regular part will be given in \cref{sub:2}). \medskip

The present paper is dedicated to the singular part of the vectorial free boundaries $\partial\Omega_U$, where $U$ is a local minimizer of the functional $J$. Before we give our main result \cref{t:main} (\cref{sub:main}), in \cref{sub:1} and \cref{sub:2}, we briefly recall some of the known properties of the solutions of the vectorial problem.\medskip

\subsection{Regularity of the vectorial free boundaries}\label{sub:1} Let $U$ be a local minimizer of $J$ in the open set $D\subset\R^d$. Then, the function $U$ is locally Lipschitz continuous in $D$, the set 
$\Omega_U:=\big\{|U|>0\big\}$
is an open subset of $D$ and the free boundary $\partial\Omega_U\cap D$ can be decomposed in the following three disjoint sets (see for instance \cite{MTV2}): 
\begin{itemize}
\item The regular part $\text{\rm Reg}(\partial\Omega_U)$ is the set of points on $\partial\Omega_U\cap D$ at which the Lebesgue density of $\Omega_U$ is precisely $\sfrac12$. It is now known that $\text{\rm Reg}(\partial\Omega_U)$ is an open subset of $\partial\Omega_U$ and is locally a $C^{1,\alpha}$-regular $(d-1)$-dimensional manifold (see \cite{KL2}, \cite{MTV2} and \cite{DST}, and also \cite{spve} for the two-dimensional case). \smallskip
\item The one-phase singular part $\text{\rm Sing}_1(\partial\Omega_U)$ is the set of points on $\partial\Omega_U\cap D$ at which the Lebesgue density of $\Omega_U$ is a number $\ell\in\big(\sfrac12,1\big)$. In \cite{MTV1} and \cite{MTV2}, it was shown that $\text{\rm Sing}_1(\partial\Omega_U)$ is a closed subset of $\partial\Omega_U\cap D$ such that:
\begin{itemize}
\item $\text{\rm Sing}_1(\partial\Omega_U)$ is empty in dimension $d<d^\ast$;
\item $\text{\rm Sing}_1(\partial\Omega_U)$ is a discrete set if the dimension of the space is exactly $d^\ast$;
\item $\text{\rm Sing}_1(\partial\Omega_U)$ has Hausdorff dimension $d-d^\ast$ if $d>d^\ast$;
\end{itemize}
where $d^\ast$ is the smallest dimension in which there are one-homogeneous solutions of the one-phase problem with a singular free boundary (that is, which is not locally the graph of a smooth function); we recall that for the moment it is only known that $d^\ast\in\{5,6,7\}$ (see \cite{JS} and \cite{DJ}).\smallskip

\item The two-phase singular part $\text{\rm Sing}_2(\partial\Omega_U)$ is the set of points on $\partial\Omega_U\cap D$ at which the Lebesgue density of $\Omega_U$ is $1$. 
\end{itemize}	

\subsection{Blow-up limits at two-phase singular points}\label{sub:2} Let $x_0\in\partial\Omega_U\cap D$ be fixed. For every $r>0$, we consider the rescaling 
$$U_{x_0,r}(x)=\frac1rU(x_0+rx).$$
Since $U$ is Lipschitz, we know that, for every $R>0$, there is $\rho>0$ such that the family of functions $U_{x_0,r}$, $r\in(0,\rho)$ is defined and uniformly Lipschitz continuous on $B_R$. Thus, there is a decreasing sequence $r_n\to0$ such that the sequence of functions $U_{x_0,r_n}$ converges locally uniformly to a function $V:\R^d\to\R^k$, which might depend on the blow-up sequence; we will say that $V$ is a blow-up limit of $U$ at $x_0$. It is well-known (see for instance \cite{MTV1}) that any blow-up limit of $U$ is
\begin{itemize}
\item Lipschitz continuous non-constantly zero function on $\R^d$;
\item a local minimizer of $J$ in $\R^d$;
\item a one-homogeneous function on $\R^d$.	
\end{itemize}	 
In \cite{MTV2} it was proved that if $x_0$ is a two-phase singular point, $x_0\in\text{Sing}_2(\partial\Omega_U)$, then every blow-up limit of $U$ at $x_0$ is of the form 
$$V(x)=Ax\quad\text{for some $d\times k$ real matrix $A$}.$$
Moreover, again in \cite{MTV2} it was shown that, even if $V$ might a priori depend on the blow-up sequence, the rank of the matrix $A$ depends only on the point $x_0$.
 We notice that the rank is an entire number between $1$ and $\min\{k,d\}$. 

\begin{enumerate}[(i)]
\item If $\text{Rk}(x_0)$ is $1$, then there is a unit vector $\nu\in\R^d$ such that the rows of the matrix $A$ are the vectors $\alpha_1\nu,\dots,\alpha_k\nu$, where the constants $\alpha_1,\alpha_2,\dots,\alpha_k\in\R$ satisfy
\begin{equation}\label{e:sum_alpha}
\alpha_1^2+\alpha_2^2+\dots+\alpha_k^2\ge \Lambda.
\end{equation}
\item If $1<\text{Rk}(x_0)\le k$, then any blow-up of $U$ at $x_0$ is of the form $Ax$, where the matrix $A=(a_{ij})_{ij}$ is such that 
$$\sum_{i,j}a_{ij}^2\ge c\Lambda,$$
where $c$ is a dimensional constant. 
\end{enumerate}
\begin{remark}
It was shown in \cite{MTV2} than if $A$ is a matrix of the form (i), for which \eqref{e:sum_alpha} holds, then the linear function $U(x)=Ax$ is a global minimizer of $J$ (that is, a local minimizer in $\R^d$). For what concerns the point (ii), classifying the matrices of rank higher than one, which are global minimizers is currently an open problem. 
\end{remark}	
  	  
\subsection{Stratification of $\text{\rm Sing}_2(\partial\Omega_U)$ and the main theorem.}\label{sub:main} For every $j=1,\dots,k$, we define the stratum $S_j$ as 
$$S_j=\Big\{x_0\in \text{Sing}_2(\partial\Omega_U)\ :\ \text{Rk}(x_0)=j\Big\}.$$
In \cite{MTV2} it was shown that, for every $j=1,\dots,k$, the set $S_j$ has Hausdorff dimension $d-j$. 
In this paper, we prove the following result 
\begin{theorem}\label{t:main}
Let $U:D\to\R^k$ be a local minimizer of $J$ in the open set $D\subset\R^d$. Then, for every $j=1,\dots,k$, the $j$-th stratum $S_j$ of $\,\text{\rm Sing}_2(\partial\Omega_U)$ is $(d-j)$-rectifiable and has locally finite $(d-j)$-dimensional Hausdorff measure. 
\end{theorem}
Combining this result with a suitable version af the Alt-Caffarelli-Friedman monotonicity formula we obtain 
\begin{corollary}\label{c:main2}
	Let $U:D\to\R^k$ be a local minimizer of $J$ in the open set $D\subset\R^d$.
	 Then, at $\mathcal H^{d-1}$-almost every point $x_0$ of $\ \text{\rm Sing}_2(\partial\Omega_U)$ the blow-up of $U$ is unique. 
\end{corollary}

\begin{remark}
	For the first stratum $\mathcal S_1$ the finiteness of the $(d-1)$-Hausdorff measure (but not the rectifiability) was proved in \cite{MTV2} by a different technique. 
\end{remark}	

\section{Proof of \cref{t:main}}

We introduce the quantitative stratification as defined in \cite{NaVa} (see \cite{EdEn} in our context).  The results of this section are simple modifications to our setting of nowadays well-understood results.

\begin{definition}
	Let $U$ be a minimizer of $J$ in $D$. Given a point $x\in \Sing_2(\partial\Omega_U)$, we say that $U$ is $(j,\eps)$-symmetric in $B_r(x)$ if 
	$$
	\frac{1}{r^{d-2}}\int_{B_r(x)}|U- A|^2\, dx<\eps
	$$
	for some linear function $A\colon \R^d\to \R^k$ such that $\text{\rm Rk}(A)=d-j$. The \emph{$(j,\eps)$-stratum, $S^j_\eps(U)$,} is the set of points $x \in \Sing_2(\partial \Omega_U)$ for which $U$ is not $(j+1,\eps)$-symmetric in $B_r(x)$, for every $0 < r \leq \min\{1, \text{\rm dist}(x,\partial D)\}$. 
\end{definition}

We remark that, as usual, $S^j= \bigcup_{\eps>0} S_\eps^j$, so that it will be enough to show that each $S^j_\eps$ is $j$-rectifiable. This follows from Naber and Valtorta breakthrough result \cite{NaVa} combined with the following observations.

\begin{lemma}[Monotonicity formula (see \cite{MTV1})]\label{l:weiss}
Let $U=(u_1,\dots,u_k):D\to\R^k$ be a local minimizer of $J$ in the open set $D\subset\R^d$. Then, the function 
$$\Phi(U,x_0,r):=\int_{B_1}|\nabla U_{x_0,r}|^2\,dx-\int_{\partial B_1}|U_{x_0,r}|^2\,d\HH^{d-1}+\Lambda\big|\big\{|U_{x_0,r}|>0\big\}\big|$$
is monotone non-decreasing in $r$ and we have 
\begin{equation}\label{e:mono}
\partial_r\Phi(U,x_0,r)\ge \sum_{\ell=1}^k\int_{\partial B_1}|x\cdot\nabla u_{\ell,x_0,r}(x)-u_{\ell,x_0,r}(x)|^2\,d\HH^{d-1}(x),
\end{equation}
where $u_{\ell,x_0,r}(x):=\frac1ru_\ell(x_0+rx)$. In particular, for any free boundary point $x_0\in\partial\Omega_U$ we can define the energy density 
$$\Phi(U,x_0,0):=\lim_{r\to0^+}\Phi(U,x_0,r),$$
which also coincides with the Lebesgue density of the set $\Omega_U$ at $x_0$.
\end{lemma}	
	
\begin{lemma}[Quantitative splitting]\label{l:qs}
For any $\rho, \gamma, E_0>0$ there exists $\eta_0=\eta_0(\rho,\gamma,E_0)>0$ such that the following holds. Let $U=(u_1,\dots,u_k):D\to\R^k$ be a local minimizer of $J$ in $D\subset\R^d$ and let $x\in S^j_\eps$.
Let $r>0$ and $E\in\R$, $|E|\le E_0$, be such  that 
\begin{equation}\label{e:qs-hypo}
B_{10r}(x)\subset D\qquad\text{and}\qquad \sup_{z\in B_{3r}(x)}\Phi(U,z,3r)\leq E.
\end{equation}
Then, at least one of the following alternatives hold:
\begin{enumerate}[\rm(1)]
\item either
\begin{equation}\label{e:alt1}
S^j_{\eps}\cap B_r(x) \subset \{z\in B_r(x)\,:\, \Phi(U,z, \gamma r)\geq E-\gamma \}\,,
\end{equation}
\item or, for any $\eta\leq \eta_0$, there is an affine $(j-1)$-dimensional space $L^{j-1}$ such that
\begin{equation}\label{e:alt2}
\{z\in B_r(x) \,:\, \Phi(U, z, 2\eta r) \geq E-\eta\}\subset  B_{\rho r}(L^{j-1})\,,
\end{equation}
where $B_{s}(L)$ denotes the tubular neighborhood of size $s$ around $p+L$.
\end{enumerate}
\end{lemma}

\begin{proof} Without loss of generality we can assume $x=0$ and $r=1$.

We can assume there are points $y_1,\dots, y_j \in B_1$ such that 
\begin{equation}\label{e:case1}
	y_i\notin B_\rho(p+{\rm span}\{y_1,\dots,y_{i-1},y_{i+1},\dots, y_j\})\quad \text{and}\quad \Phi(U,y_i,2\eta) \geq E-\eta  
\end{equation}	
for every $i=1,\dots,j$ and for some $\eta>0$ to be chosen later. Indeed if such points don't exist, then there is a $(j-1)$-dimensional affine plane $L^{j-1}$ such that
$$
y\in B_1\setminus B_{\rho}(L^{j-1}) \quad \Rightarrow \quad \Phi(U,y,2\eta)<E-\eta\,,
$$
which is precisely the alternative (2).

Now we claim that there exist $\eta, \beta>0$ such that if \eqref{e:qs-hypo} and \eqref{e:case1} hold, for some $E$ and $\rho>0$, then we have
\begin{gather}
\Phi(U, z, \gamma)\geq E-\gamma\qquad \text{for every}\qquad z\in B_\beta(L)\cap B_1\label{e:1}\\
S^j_\eps\cap B_1\subset B_{\beta}(L)\cap B_1 \label{e:2}\,,
\end{gather} 
where $L=p+{\rm span}\{y_1,\dots,y_j\}$, which clearly implies that alternative (1) holds and concludes the proof.

Suppose therefore that \eqref{e:1} fails. Then there are sequences $U_n, y_{n,i}, L_n, \eta_n, \beta_n, E_n$ such that $\eta_n,\beta_n\to 0$ and for each $n\in\N$ there is $x_n \in B_{\beta_b}(L_n)\cap B_1$ such that 
$$\Phi(U_n,x_n, \gamma)\leq E_n-\gamma.$$ 
Up to passing to a subsequence we have that $U_n$ converges strongly in $H^1_{loc}(B_{9})$ and locally uniformly in $B_{9}$ to a function $V:B_{9}\to\R$, which is a local minimizer of $J$ in $B_{9}$. Moreover, we can also suppose that
$$
E_n\to E,\quad y_{n,i}\to y_i,\quad L_n\to L,\quad x_{n}\to x\in \overline B_1\cap L \,.
$$
 By the contradiction assumption and by the continuity of $\Phi$ for fixed radius we get
\begin{equation}\label{e:cont1}
\Phi(V,x,0)\le \Phi(V,x,\gamma)=\lim_{n\to\infty}\Phi(U_n,x_n,\gamma)\leq E-\gamma\,.
\end{equation}
On the other hand, we notice that by \eqref{e:case1} and the contradiction asssumption we have 
$$\Phi(U_n,y_{n,i},2\eta_n) \geq E_n-\eta_n\quad\text{for every}\quad n\ge 1.$$
Thus, for every fixed $\rho>0$, 
$$\Phi(V, y_i, \rho)=\lim_{n\to\infty}\Phi(U_n,y_{n,i},\rho)\ge \lim_{n\to\infty}\Phi(U_n,y_{n,i},2\eta_n) \geq \lim_{n\to\infty}(E_n-\eta_n)=E,$$
and passing to the limit as $\rho\to0$, we get
$$
\Phi(V, y_i, 0)\geq E \quad \text{for every}\quad i\in\{1,\dots j\}\,.
$$
On the other hand, by hypothesis, 
$$\sup_{z\in B_3} \Phi(U_n,z,3)\leq E_n\quad\text{for every}\quad n\ge 1.$$
Thus, using the coninuity of $\Phi$ in the first two variables (for fixed radius $r=3$), we get 
$$\sup_{z\in B_1} \Phi(V,z,3)\leq E.$$
Now the monotonicity formula for $V$ at the points $y_1,\dots,y_j$ implies that 
$$\Phi(V,y_i,r)= E\quad\text{for every}\quad i\in\{1,\dots,j\}\quad\text{and every}\quad 0\le r<3.$$
Thus, by \cref{l:weiss} implies that $V$ is one-homogeneous with respect to each of the points $y_1,\dots,y_j$, which means that $V$ is one-homogeneous and independent of $L$. Now since $x\in L$, this implies that $\Phi(V,x,0) =E$, which is a contradiction with \eqref{e:cont1}.\smallskip

Thus, we know that there are $\eta_0>0$ and $\beta>0$ such that if \eqref{e:qs-hypo} and \eqref{e:case1} hold, for some $E\in(-E_0,E_0)$, $\rho>0$ and $\eta\le\eta_0$, then also \eqref{e:1} holds. Thus, it remains to show that with this choice of $\beta$ we have also \eqref{e:2}. Arguing again by contradiction, we suppose that \eqref{e:2} fails for a  sequence $U_n, y_{n,i}, L_n$ and $x_n \in \Big(S^j_\eps(U_n)\setminus B_{\beta}(L_n)\Big)\cap B_1$ for which \eqref{e:qs-hypo}, \eqref{e:case1} and \eqref{e:1} hold for some $\eta_n\to0$ (with fixed $\rho,\beta,\gamma$, and $E$).  Then, by the same argument as above, we obtain a $1$-homogeneous function $V$ (the limit of $U_n$) invariant on $L$ (the limit of $L_n$) and a point $x\in \Big(S_{\eps}^j(V)\setminus B_\beta(L)\Big)\cap B_1$, which is the limit of $x_n$. This implies that any blow-up $W$ of $V$ at $x$ has $j+1$ symmetries. In fact, since $V$ is invariant with respect to every $y\in L$, so is $W$; moreover, since $V$ is $1$-homogeneous and $x\neq0$, its blow-ups at $x$ are invariant in the direction of $x$. Thus, $U_n$ is $(j+1,\eps)$-symmetric at $x_n$, for $n$ sufficiently large, which is a contradiction with $x_n\in  S_{\eps}^j(U_n)$.
\end{proof}

\begin{lemma}[Effective control of the $\beta$ number-$L^2$ estimate]\label{l:ec} Let $U=(u_1,\dots,u_k):D\to\R^k$ be a local minimizer of $J$ in the open set $D\subset\R^d$ and let $x_0\in S^j_\eps(U)$. There is a $\delta=\delta(d,\eps)>0$ such that if 
$$
\Phi(U,x_0,8r)-\Phi(U,x_0,\delta r)<\delta\,,
$$
then for any finite Borel measure $\mu$ the following estimate holds
$$
\beta_\mu^j(x,r)^2\leq \frac{C(d,\eps)}{r^j} \int_{B_r(x)} \big( \Phi(U,y,8r)-\Phi(U,y,r) \big) \,d\mu(y)\,,
$$
where $\beta_\mu^j(x,r)$ is the Jones's $\beta$ number defined by
$$
\beta_\mu^j(x,r)^2=\inf_{p+L^j} \frac1{r^{j+2}} \int_{B_r(x)} \text{\rm dist}\,(y, p+L^j)^2\,d\mu(y)\,.
$$
where the infimum is taken over all affine $j$-dimensional planes $p+L^j$.
\end{lemma}

\begin{proof} Let $p_{x,r}$ be the barycenter of $\mu$ in $B_r(x)$, that is
	$$
	p_{x,r}:=\frac{1}{\mu(B_r(x))} \int_{B_r(x)}y\,d\mu(y)=\fint_{B_r(x)} y\,d\mu(y)\,.
	$$ 
	Consider the symmetric positive semi-definite bilinear form $B\colon \R^d\times \R^d\to \R$ defined by 
	$$
	B_x(v,w):=\fint_{B_r(x)} \left((y-p_{x,r})\cdot v\right) \, \left((y-p_{x,r})\cdot w\right)\,d\mu(y)\qquad \forall v,w \in \R^d\,.
	$$
	By standard linear algebra, there exists an orthonormal basis of vectors $\{v_1, \dots , v_d\} \subset \R^d$ which diagonalizes the bilinear form $B_x$, that is
	$$
	B_{x}(v_i,v_j)=\delta_{ij}\,\lambda_i\qquad \text{where}\qquad 0\leq \lambda_d\leq\dots\leq \lambda_1\,. 
	$$
	If we denote with $L^j_{\mu}(x,r)$ the plane realizing the infimum in the definition of $\beta^j_\mu(x,r)$, it is then easy to check that
	\begin{equation}\label{e:L2}
	L^j_\mu=p_{x,r}+{\rm span}\{v_1,\dots,v_j\}\qquad \text{and}\qquad \beta^j_\mu(x,r)^2=\frac{\mu(B_r(x))}{r^j}\sum_{i=j+1}^d\lambda_i\ .
	\end{equation}
	Moreover, since the barycenter $p_{x,r}$ satisfies the equation
	\begin{equation}\label{e:barycenter}
	\int_{B_r(x)}(y-p_{x,r})\,d\mu(y)=0,
	\end{equation} 
	for every $i=1,\dots,d$, we have
	\begin{align}\label{e:silly1}
	\lambda_i\, v_i
			&=\sum_{j=1}^d\left(\fint_{B_r(x)}\left((y-p_{x,r})\cdot v_i\right) \left((y-p_{x,r})\cdot v_j\right)\,d\mu(y)\right)v_j\notag\\
		&=\fint_{B_r(x)}\left((y-p_{x,r})\cdot v_i\right) \left[\sum_{j=1}^d\left((y-p_{x,r})\cdot v_j\right)\,v_j\right]  \,d\mu(y)\notag\\
		&=\fint_{B_r(x)}\left((y-p_{x,r})\cdot v_i\right) (y-p_{x,r})\,d\mu(y) \notag\\
		&=\fint_{B_r(x)}\left((y-p_{x,r})\cdot v_i\right) y\,d\mu(y)\,.
	\end{align}
	Next, for each component $u_\ell$ of $U$, we compute
	\begin{align*}
	\lambda_i\,(v_i\cdot \nabla u_\ell(z))
		&=\fint_{B_r(x)} \left( (y-p_{x,r}) \cdot v_i   \right)\,\left(y \cdot \nabla u_\ell(z)   \right)\,d\mu(y)\\
		&=\fint_{B_r(x)} \left( (y-p_{x,r}) \cdot v_i   \right)\,\left( u_\ell(z)-(z-y) \cdot \nabla u_\ell(z)   \right)\,d\mu(y)\,,
	\end{align*}
	where in the second equality we used again \eqref{e:barycenter}. Using H\"older inequality we deduce
	\begin{align*}
	\lambda_i^2\,|v_i\cdot \nabla u_\ell(z)|^2	
		&\leq \fint_{B_r(x)} \left( (y-p_{x,r}) \cdot v_i   \right)^2\,d\mu(y)\, \fint_{B_r(x)}\left(u_\ell(z)-(z-y) \cdot \nabla u_\ell(z)  \right)^2\,d\mu(y)\\
		&=\lambda_i \,\fint_{B_r(x)}\left(u_\ell(z)-(z-y) \cdot \nabla u_\ell(z)    \right)^2\,d\mu(y)\,.
	\end{align*}
	Summing over the components of $U$, we conclude 
	\begin{equation}\label{e:L21}
	\lambda_i\,\sum_{\ell=1}^k |\nabla u_\ell(z)\cdot v_i|^2	\leq \sum_{\ell=1}^k \fint_{B_r(x)}\left(u_\ell(z)-(z-y) \cdot \nabla u_\ell(z)    \right)^2\,d\mu(y)\,.
	\end{equation}
	Next, we set $A_{s,t}(x):=B_t(x)\setminus B_s(x)$ and compute
	\begin{align}
	\lambda_i &r^{-d-2} \int_{A_{3r,4r}(x)}\sum_{\ell=1}^k \left(\nabla u_\ell(z)\cdot v_i\right)^2\,dz \notag\\	
		&\leq r^{-d-2} \int_{A_{3r,4r}(x)}  \sum_{\ell=1}^k \fint_{B_r(x)}\left(u_\ell(z)-(z-y) \cdot \nabla u_\ell(z)    \right)^2\,d\mu(y)\,dz\notag\\
		&\leq C_k\, \fint_{B_r(x)}\sum_{\ell=1}^k \int_{A_{3r,4r}(x)} \left(u_\ell(z)-(z-y) \cdot \nabla u_\ell(z)    \right)^2\,|z-y|^{-d-2}dz\,d\mu(y)\notag\\
		&\leq C_k\, \fint_{B_r(x)}    \left(\Phi(U,x,8r)-\Phi(U,x,r)  \right)      \,   d\mu(y)
	\end{align}
	where in the second inequality we used \eqref{e:L21} and in the last inequality we used \eqref{e:mono}.
	
	Next we claim that there is $\delta=\delta(\eps,d)$, $c=c(d,\eps)>0$ such that for any orthonormal family of vectors $\{v_1,\dots,v_{j+1}\}$ we have
	\begin{equation}\label{e:L22}
	c\leq \frac1{r^{d+2}} \int_{A_{3r,4r}(x)} \sum_{i=1}^{j+1}\sum_{\ell=1}^k(v_i\cdot \nabla u_\ell(z))^2\,dz\,.
	\end{equation}
	
It is enough to show this when $r=1$ and $x=0$, so we assume by contradiction that there is a sequence of minimizers $U_n=(u_{n,1},\dots,u_{n,k})$ and othonormal systems $\{v_1^n,\dots,v_{j+1}^n\}$ such that 
	\begin{equation}\label{e:bla1}
 \int_{A_{3,4}(0)} \sum_{i=1}^{j+1}\sum_{\ell=1}^k(v_i^k\cdot \nabla u_{n,\ell}(z))^2\,dz \leq \frac 1n\,,
\end{equation}
and moreover 
\begin{equation}\label{e:bla2}
\Phi(U_n,0,8)-\Phi(U_n,0,\sfrac1{n})<\frac 1n\,.
\end{equation}
Up to passing to a subsequence we can assume that $U_n\to V$ strongly in $W^{1,2}$, with $V$ a minimizer. By \eqref{e:bla2} we have that $V$ is $1$-homogeneous, by \eqref{e:bla1} we have that  $v_i^n\cdot \nabla u_{n,\ell}(z)=0$ for every $z\in A_{3,4}(0)$, $\ell=1,\dots,k$ and $i=1,\dots,j+1$, which implies that $V$ is $j+1$ symmetric in $B_8$ (as it is $1$-homogeneous). This is a contradiction with the hypothesis $0\in S^j_\eps(U_n)$ for every $n$.

Finally, combining \eqref{e:L2}, \eqref{e:L21} and \eqref{e:L22} we conclude
\begin{align*}
\beta^j_\mu(x,r)^2
	& \leq \frac{\mu(B_r(x))}{r^j} \, k \, \lambda_{j+1} \\
	&\leq k\frac{\mu(B_r(x))}{c\,r^j} \sum_{i=1}^{j+1} \frac{\lambda_{i}}{r^{d+2}} \int_{A_{3r,4r}(x)} \sum_{i=1}^{j+1}\sum_{\ell=1}^k(v_i\cdot \nabla u_\ell(z))^2\,dz\\
	& \leq \frac{C(d,\eps)}{r^j} \int_{B_r(x)}    \left(\Phi(U,x,8r)-\Phi(U,x,r)  \right)      \,   d\mu(y)\,,
\end{align*}
as desired.
\end{proof}

\begin{proof}[Proof of \cref{t:main}]
The proof of \cref{t:main} follows by combining Lemmas \ref{l:weiss}, \ref{l:qs} and \ref{l:ec} with the work of Naber-Valtorta \cite{NaVa} (see also \cite{Ed} for a more concise presentation of the fact that the three lemmas imply the desired result).	
\end{proof}

\section{Uniqueness of the blow-up limits. Proof of \cref{c:main2}}
In order to prove \cref{c:main2} we will need the following lemma.
\begin{lemma}
Let $U:D\to\R^k$ be a local minimizer of $J$ in the open set $D\subset\R^k$. Let $x_0\in \text{\rm Sing}_2(\partial\Omega_U)$. Let $Ax$ and $Bx$ are two linear functions obtained as blow-up limits of $U$ at $x_0$, then 
$$|A^t\sigma|=|B^t\sigma|\qquad\text{for every}\qquad\sigma\in\R^k.$$
\end{lemma}
\begin{proof}
Suppose, without loss of generality, that $x_0=0$. For every vector $\sigma\in\R^k$, consider the function 
$$U\cdot\sigma:D\to\R.$$
Since $U$ is harmonic, where it is non-zero, we have that also the function $\sigma\cdot U$ is harmonic on the sets $\{\sigma\cdot U>0\}$ and $\{\sigma\cdot U<0\}$. Thus, by the Alt-Caffarelli-Friedman monotonicity formula (see \cite{alcafr}), we have that the quantity
$$\Psi(U,\sigma,r)=\frac1{r^4}\left(\int_{B_r\cap \{\sigma\cdot U>0\}}\frac{|\nabla (\sigma\cdot U)|^2}{|x|^{d-2}}\,dx\right)\left(\int_{B_r\cap \{\sigma\cdot U<0\}}\frac{|\nabla (\sigma\cdot U)|^2}{|x|^{d-2}}\,dx\right)$$
is non-decreasing in $r$. In particular, the limit 
$$\lim_{r\to0}\Psi(U,\sigma,r),$$
exists. Suppose now that the blow-up sequence $U_{0,r_n}$ converges (locally uniformly and strongly in $H^1_{loc}$) to a blow-up limit of the form $Ax$, where $A$ is $d\times k$ matrix. Then, 
since 
$$\Psi(U,\sigma,r)=\Psi(U_{0,r},\sigma,1),$$
we have that
$$\lim_{r\to0}\Psi(U,\sigma,r)=\left(\int_{B_1\cap \{\sigma\cdot Ax>0\}}\frac{| A^t\sigma|^2}{|x|^{d-2}}\,dx\right)\left(\int_{B_1\cap \{\sigma\cdot Ax<0\}}\frac{|A^t\sigma|^2}{|x|^{d-2}}\,dx\right),$$
which concludes the proof.
\end{proof}

\begin{proof}[\bf Proof of \cref{c:main2}]
By \cref{t:main} we know that at almost every point $x_0\in \text{Sing}_2(\partial\Omega_U)$ there is a unique tangent plane 
$T=\{x\ :\ x\cdot\nu=0\}$ (where $\nu\in\R^d$ is a unit vector) to $\partial\Omega_u$. Let now $U_{x_0,r_n}$ be a blow-up sequence converging to a blow-up limit of the form $Ax$. Since $\partial\{|U_{x_0,r_n}|>0\}$ converges in the Hausdorff distance in $B_1$ to $\partial\{|Ax|>0\}$, we have that $\text{Ker}A$ is precisely the tangent plane $T$. Thus, $A=\eta\otimes\nu$ for some vector $\nu\in\R^k$. Thus, for any vector $\sigma\in\R^k$, we have that 
$|A^t\sigma|=|\eta\cdot\sigma|$. As a consequence, the vector $\eta$ (and thus the matrix $A$) does not depend on the blow-up sequence. 	
\end{proof}

\end{document}